\documentclass[11pt,a4paper]{article}
\usepackage[OT1]{fontenc}
\usepackage[english]{babel}
\usepackage{amsfonts}
\usepackage{amsthm}

\usepackage{graphicx}

\def\q{ {\cal Q} }
\def\b{ {\cal B} }

\def\t{ {\cal T} }
\def\s{ {\cal S} }
\def\e{ {\cal E} }
\def\o{ {\cal O} }
\def\p{ {\cal P} }

\def\DD{\mathbb{D}}
\def\TT{\mathbb{T}}
\def\rr{{\bf r}}
\def\nn{{\bf n}}
\def\gg{{\bf g}}

\newtheorem{teo}{Theorem}[section]
\newtheorem{prop}[teo]{Proposition}
\newtheorem{lem}[teo]{Lemma}
\newtheorem{coro}[teo]{Corollary}

\theoremstyle{definition}
\newtheorem{rem}[teo]{Remark}

\newtheorem{quest}[teo]{Question}
\newtheorem{claim}[teo]{Claim}

\title{Symmetries and reflections from composition operators in the disk}
\author{E. Andruchow, G. Corach, L. Recht}

\begin{document}

\maketitle 

\begin{abstract}
The set $\q$ of reflections (i.e., operators $C$ such that $C^2=I$) in a C$^*$-algebra   is a geometric space which has been the object of several investigations, and is an important tool in the study of these algebras. In this paper we consider a special class of reflections,    the composition operators $C_a$ acting on the Hardy space $H^2$ of the unit disk, given by $C_af=f\circ\varphi_a$, where 
$$
\varphi_a(z)=\frac{a-z}{1-\bar{a}z},
$$ 
for $|a|<1$. These operators are indeed reflections, because $\varphi_a\circ\varphi_a=id$. We study their eigenspaces $N(C_a\pm I)$, their relative position (i.e., the intersections between these spaces and their orthogonal complements for $a\ne b$ in the unit disk)  and the symmetries induced by $C_a$ and these eigenspaces
 \end{abstract}

\bigskip



\section{Introduction}
Let $\mathbb{D}=\{z\in\mathbb{C}:|z|\le 1\}$ be the unit disk and $\mathbb{T}=\{z\in\mathbb{C}: |z|=1\}$ the unit circle. Consider the analytic automorphisms $\varphi_a$ which map $\mathbb{D}$ onto $\mathbb{D}$  of the form
$$
\varphi_{a}(z)=\frac{a-z}{1-\bar{a}z},
$$
for $a\in\DD$. Save for a constant of modulus one, all automorphisms of the disk are of this form.  Note the fact that $\varphi_a(\varphi_a(z))=z$.
This implies that the composition operators induced by these automorphisms are {\it reflections} (i.e., operators $C$ such that $C^2=I$). Namely, let $H^2=H^2(\DD)$ be the Hardy space of the disk, i.e.
$$
H^2=\{f:\DD\to\mathbb{C}: f(z)=\sum_{n=0}^\infty a_n z^n \hbox{ with } \sum_{n=0}^\infty |a_n|^2<\infty\}.
$$
Then, an analytic map $\varphi:\DD\to\DD$  induces the (bounded linear, see \cite{cowenmccluer}) operator $C_\varphi:H^2\to H^2$, 
$$
C_\varphi f=f\circ\varphi.
$$
In particular, for $a\in\DD$, the operator $C_a:=C_{\varphi_a}$ satisfies $C_a^2=I$, the identity operator. 

The space $\q$ of reflections in a C$^*$-algebra and its subset $\p$ of selfadjoint elements (called symmetries) have been the object of several studies over the  years (see for instance \cite{kovarik}, \cite{pr}, \cite{cpr}, \cite{wilkins}, \cite{chung} focusing on geometric properties, or \cite{zemanek}, \cite{brown}, \cite{phillips}  on metric and  topological aspects). Note that reflections or symmetries may appear  under the guise of idempotents or projections: $C$ is a reflection (resp., a symmetry) if and only if $\frac12(C+I)$ is an idempotent (resp., a projection). This paper is our first attempt to understand the geometry of a special class of reflections, namely the operators $C_a$ indexed by $a\in\mathbb{D}$. Further efforts will be focused in understanding the interplay of the geometry of $\mathbb{D}$ and that of $Q$.

The eigenspaces of  $C_a$ are  
$$
N(C_a-I)=\{f\in H^2: f\circ \varphi_a=f\} \ \hbox{ and }\   N(C_a+I)=\{g\in H^2: g\circ \varphi_a=-g\},
$$
which verify that $N(C_a-I)\dot{+}N(C_a+I)=H^2$.
Here $\dot{+}$ means direct (non necessarily orthogonal) sum, we reserve the symbol $\oplus$ for orthogonal sums.

Reflections which additionally are selfadoint are called {\it symmetries}: $S$ is a symmetry if $S=S^*=S^{-1}$. 
Associated to a reflection $C$, there are three natural symmetries: $\rr(C)$, $\nn(C)$ and $\rho(C)$. The first two correspond to the decompositions  $H^2=N(C-I)\oplus N(C-I)^\perp$ 
and $H^2=N(C+I)\oplus N(C+I)^\perp$ respectively. The third is of differential geometric nature, and is described below. The aim of this paper is the study of the operators $C_a$  for $a\in\DD$, the description of their eigenspaces, their relative position, and the induced symmetries.
In this task, it will be important the role of the unique fixed point $\omega_a$ of $\varphi_a$ inside the disk. Namely,
\begin{equation}\label{punto fijo omega}
\omega_a:=\frac{1}{\bar{a}}\{1-\sqrt{1-|a|^2}\} \ \hbox{ if } a\ne 0, \ \hbox{ and } \omega_0=0.
\end{equation}

The contents of the paper are the following. In Section 2 we recall basic facts on the manifolds of reflections and symmetries, in particular the condition for existence of geodesics between points in the latter space. In Section 3 we state basic formulas concerning the operators $C_a$. In Section 4 we characterize the symmetries $\rho_a$, obtained as the unitary part of the polar decomposition of $C_a$. For this task, we use Rosenblum's computation for the spectral measure of a selfadjoint Toeplitz operator \cite{rosenblum}. Using a result by E. Berkson \cite{berkson}, we show that locally, the map $a\mapsto \rho_a$ ($a\in\DD$) is injective (it remains unanswered wether it is globally injective in the disk $\DD$). We also obtain formulas for the range an nullspace symmetries of $C_a$, and a power series expansion for $\rho_a$. The rest of the paper is devoted to the study of the eigenspaces of $C_a$, and their relative position. If $a=0$, then  the fixed point of $\varphi_0$  is $\omega_0=0$ and $C_0$ is the reflection (and symmetry) $f(z)\mapsto f(-z)$. Thus the eigenspaces of $C_0$ are the subspaces $\e$ and $\o$ of even and odd functions of $H^2$. It is elemenatry to see that for arbitrary $a\in\DD$, the eigenspaces of $C_a$ are
$$
N(C_a-I)=C_{\omega_a}(\e) \ \  \hbox{ and } \ \  N(C_a+I)=C_{\omega_a}(\o).
$$
We then analyze the position of these eigenspaces for $a\ne b$. For instance (Theorem \ref{teo 54}),
$$
N(C_a-I)\cap N(C_b-I)=\mathbb{C}1 \hbox{ and } N(C_a+I)\cap N(C_b+I)=\{0\}.
$$
The computations of the intersections involving the orthogonal of these spaces is more cumbersome, and we are only able to do it in the special case when $b=0$ (Theorem \ref{0 y a}). These facts, which are stated in Section 5,  are used in Section 6 to show which of these eigenspaces are conjugate with the exponential of the Grassmann manifold of $H^2$.

This work was supported by the grant PICT2019 0460, from ANPCyT, Argentina.

\section{Preliminaries, on reflections and symmetries}
Denote the set of reflections by
$$
\q=\{T\in\b(H^2): T^2=I\}.
$$
The set $\q$ has rich geometric structure (see for instance \cite{cpr}): is it a homogeneous C$^\infty$ submanifold of $\b(H^2)$, carrying the action of the  group $Gl(H^2)$ of invertible operators in $H^2$: 
$$
G\cdot T=GTG^{-1}, \ T\in\q, G\in Gl(H^2).
$$
The set $\p$ of {\it selfadjoint} reflections, or {\it symmetries}, is
$$
\p=\{ V\in\q: V^*=V\}.
$$
Note that a symmetry $V$ is a selfadjoint unitary operator. Reflections and symmetries can be viewed alternatively as oblique and orthogonal projections, respectively. A reflection $T$ gives rise to an idempotent (or oblique projection) with range equal to the eigenspace $\{f\in H^2: Tf=f\}$: $Q_+=\frac12(I+T)$ (and another with range equal to the other eigenspace $\{g\in H^2: Tg=-g\}$ of $T$: $Q_-=\frac12(I-T)$). If $S$ is a symmetry, the corresponding idempotents $P_+$ and $P_-$ are orthogonal projections. 

The set $\p$, in turn, can be regarded as the Grassmann manifold of $H^2$: to each reflection $V$ corresponds a unique projection $P_+=\frac12(I+V)$ and a unique subspace  $\s$ such that $R(P_+)=\s$. The geometry of the Grassmann manifold in this operator theoretic context was developed in \cite{cpr}, \cite{pr}: $\p$ is presented as a homogeneous space of the unitary group (as in the classical finite dimensional setting), with a linear reductive connection and a Finsler metric. In \cite{p-q} the necessary and sufficient condition for the existence of a geodesic of this connection  between two subspaces $\s$ and $\t$ was stated: namely, that
\begin{equation}\label{condicion geodesica}
\dim (\s\cap\t^\perp)=\dim (\s^\perp\cap\t).
\end{equation}
Moreover, the geodesic is of the form $\delta(t)=e^{itX}\s$, for $X^*=X$ co-diagonal with respect to both $\s$ and $\t$:
$$
X(\s)\subset\s^\perp \ \ \hbox{ and } \  \ X(\t)\subset\t^\perp.
$$ 
The geodesic can be chosen of minimal length for the Finsler metric (see \cite{pr}, \cite{cpr}, \cite{p-q}). This latter condition amounts to finding $X$ such that $\|X\|\le\pi/2$.  

The condition for the existence of a {\it unique}  minimal geodesic  (up to reparametrization) was given:
\begin{equation}\label{unicidad geodesica}
\s\cap\t^\perp=\{0\}=\s^\perp\cap\t.
\end{equation}
In this case, the exponent $X=X_{\s,\t}$ is unique with the above mentioned conditions ($X_{\s,\t}$ selfadjoint, codiagonal with respect to $\s$ and $\t$, with norm less or equal then $\pi/2$, satisfying  $e^{iX_{\s,\t}}\s=\t$.).

In this paper we shall examine existence and uniqueness of geodesics of the Grassmann manifold of $H^2$, for the eigenspaces of $C_a$. 

One of the remarkable features of the space $\q$ is the several natural  projection maps that it has onto  $\p$.
The natural projection maps $\q\to\p$ are the range, nullspace and unitary part in the polar decomposition:
\begin{enumerate}
\item
The {\it range map} $\rr$, which maps $T\in\q$ to the symmetry $\rr(T)=2P_{R(Q_+)}-I$, i.e. the symmetry which is the identity on $R(Q_+)=\{f\in H^2: Tf=f\}$. We recall the formula for the orthogonal projection $P_{R(Q)}$ onto the range $R(Q)$ of an idempotent $Q$ (see for instance \cite{ando}):
$$
P_{R(Q)}=Q(Q+Q^*-I)^{-1}.
$$
Then 
$$
P_{R(Q_+)}=\frac12(I+T)\{\frac12(I+T)+\frac12(I+T^*)-I\}^{-1}=(I+T)\{T+T^*\}^{-1},
$$
and therefore
\begin{equation}\label{rango}
\rr(T)=2(I+T)\{T+T^*\}^{-1}-I=(2I+T-T^*)\{T+T^*\}^{-1}.
\end{equation}
 \item
The {\it null-space} map $\nn$, which maps $T\in\q$ to the symmetry which is the identity on $R(Q_-)=\{g\in H^2: Tg=-g\}$, which by similar computations is given by
\begin{equation}\label{nucleo}
\nn(T)=2(T-I)\{T+T^*\}^{-1}-I=(T-T^*-2I)\{T+T^*\}^{-1}.
\end{equation}.
\item
The {\it unitary part} $\rho$ {\it in the polar decomposition}, which maps $T$ to 
\begin{equation}\label{parte unitaria}
\rho(T)=T(T^*T)^{-1/2},
\end{equation}
the unitary part in the polar decomposition $T=\rho(T)(T^*T)^{1/2}$. We refer the reader to \cite{cpr} for the properties of this element $\rho(T)$. Among them, the most remarkable, that $\rho(T)$ is a symmetry. We shall recall the other properties of $\rho(T)$ in due course. Note, for instance, that $(T^*T)^{-1}=TT^*$, so that
$$
(T^*T)^{-1/2}=(TT^*)^{1/2}.
$$
\end{enumerate}
Notice the following formulas:
\begin{prop}\label{r(T) y n(T)}
Let $T\in\q$ then
$$
\rr(T)=2(I+T)(T^*T+I)^{-1} \ \hbox{ and }  \ \ \nn(T)=2(I-T)(T^*T+I)^{-1}.
$$
\end{prop}
\begin{proof}
Let $T=\rho(T)|T|$ be the polar decomposition. It is a straightforward computation  (or see \cite{cpr}) that $|T|\rho(T)=\rho(T)|T^*|$. Also it is eassy to see that since $T^2=I$, $|T^*|=|T|^{-1}$. Then 
$$
T+T^*=\rho(T)|T|+|T|\rho(T)=\rho(T)(|T|+|T^*|)=\rho(T)(|T|+|T|^{-1}).
$$
Using again that $|T|\rho(T)=\rho(T)|T|^{-1}$ (and therefore also that $\rho(T)|T|=|T|^{-1}\rho(T)$), we have that $\rho(T)$ commutes with $|T|+|T|^{-1}$. Then
$$
(T+T^*)^{-1}=\rho(T)(|T|+|T|^{-1})^{-1}=(|T|+|T|^{-1})^{-1}\rho(T).
$$
By an elementary functional calculus argument, we have that $ (|T|+|T|^{-1})^{-1}=|T|(|T|^2+I)^{-1}$. Then 
$$
(T+T^*)^{-1}=\rho(T)|T|(|T|^2+I)^{-1}=T(|T|^2+I)^{-1}.
$$
Thus,
$$
\rr(T)=2(T+I)T(|T|^2+I)^{-1}=2(I+T)(|T|^2+I)^{-1},
$$
and similarly
$$
\nn(T)=2(T-I)T(|T|^2+I)^{-1}=2(I-T)(|T|^2+I)^{-1}.
$$
\end{proof}
 We shall return to these formulas for the case $T=C_a$ later, after we further characterize $|C_a|$.

\section{The operators $C_a$}
It is not a trivial task to compute the adjoint of a composition operator, however, for the special case of automorphisms of the disk, it was shown by Cowen \cite{cowen} (see also \cite{cowenmccluer}) that
$$
C_a^*=(C_{\varphi_a})^*=M_{\frac{1}{1-\bar{a}z}} C_a (M_{1-\bar{a}z})^*,
$$
where, for  a bounded analytic function $g$ in $\DD$,  $M_g$ denotes the multiplication operator. Equivalently,
\begin{equation}\label{cea*}
C_a^*=M_{\frac{1}{1-\bar{a}z}} C_a-a M_{\frac{1}{1-\bar{a}z}} C_a (M_z)^*,
\end{equation}
where $(M_z)^*$  (or co-shift) is the adjoint of the shift operator $S=M_z$.

In order to characterize the polar decomposition of $C_a$, it will be useful to compute
$C_aC_a^*$. Note that, for $f\in H^2$, after straightforward computations, 
\begin{equation}\label{ceacea*}
C_aC_a^*f(z)=\frac{1-\bar{a}z}{1-|a|^2}\{f(z)-a\frac{f(z)-f(0)}{z}\}.
\end{equation}
Also note how $C_a$ relates to the shift operator
$$
S=M_z:H^2\to H^2, \ Sf(z)=zf(z), \ \hbox{ with adjoint } S^*f(z)=\frac{f(z)-f(0)}{z}:
$$
\begin{equation}\label{ceacea* shift}
C_aC_a^*=\frac{1}{1-|a|^2}(1-\bar{a}S)(I-aS^*).
\end{equation}
For $a\in\DD$, denote by $k_a$ the Szego kernel: for $f\in H^2$, $\langle f , k_a\rangle=f(a)$, i.e.,
\begin{equation}\label{szego}
k_a(z)=\displaystyle{\frac{1}{{1-\bar{a}z}}}.
\end{equation}
\begin{rem}\label{rem szego}
Note the  fact that 
$$
C_aC_a^*(k_a)=1.
$$
Indeed, this follows after a straightforward computation. Therefore, we have also that 
$$
C_a^*C_a(1)=k_a.
$$
\end{rem}
For $a\in\DD$, denote by
\begin{equation}\label{def rho}
\rho_a=\rho(C_a).
\end{equation}
Note that if $a=0$, $\varphi_0(z)=-z$ and $C_0f(z)=f(-z)$ is a symmetry, thus $C_0^*=C_0$, $|C_0|=I$ and $\rho_0=C_0$. 

Returning to the characterization of the modulus of $C_a$, we have that 
\begin{lem}
With the current notations,
$$
|C_a^*|=\frac{1}{\sqrt{1-|a|^2}}|I-aS^*|.
$$
and 
$$
\rho_a=\frac{1}{\sqrt{1-|a|^2}}C_a|I-aS^*|.
$$
\end{lem}
\begin{rem}
There is another symmetry related to $C_a$. In the book \cite{cowenmccluer} (Exercise 2.1.9:), it is stated that for $a\in\DD$, if we put 
$$
\psi_a(z)=\frac{\sqrt{1-|a|^2}}{1-\bar{a}z}=\sqrt{1-|a|^2}\ k_a=\frac{k_a}{\|k_a\|_2},
$$
then the operator $W_a\in\b(H^2)$, $W_a=M_{\psi_a} C_a$. i.e.,
$$ 
W_af(z)=\psi_a(z) f(\varphi_a(z))
$$
is a unitary operator.
In fact, it is straightforward to verify that $W_a^2=I$, i.e., $W_a$ is a symmetry. 
\end{rem}
Note the relationship between $\rho_a$ and $W_a$:
\begin{equation}\label{relacion ce wa}
C_a=\frac{1}{\sqrt{1-|a|^2}} M_{1-\bar{a}z} W_a=\frac{1}{\sqrt{1-|a|^2}}(1-\bar{a}S)W_a.
\end{equation}
It follows that the symmetry $W_a$ intertwines $C_aC_a^*$ and $C_a^*C_a$:
$$
C_a^*C_a=\frac{1}{1-|a|^2} W_a(I-a S^*)(I-\bar{a}S)W_a=W_a(C_a^*C_a)W_a,
$$
thus
\begin{equation}\label{modulo cea}
|C_a|=\frac{1}{\sqrt{1-|a|^2}}W_a |I-\bar{a}S|W_a=W_a|C_a^*|W_a,
\end{equation}
and $|C_a|^{-1}=\sqrt{1-|a|^2}W_a |I-\bar{a}S|^{-1}W_a$.

\begin{rem}
Note that $C_a=\rho_a|C_a|$ implies that 
$$
C_aC_a^*=\rho_a|C_a|^2\rho_a=\rho_a C_a^*C_a\rho_a.
$$
Then 
$$
W_a\rho_a C_a^*C_a (W_a\rho_a)^*=W_a\rho_aC_a^*C_a\rho_aW_a=W_aC_aC_a^*W_a=C_a^*C_a,
$$
i.e., $W_a\rho_a$ commutes with $C_a^*C_a$ (and therefore also with its inverse $C_aC_a^*$).
\end{rem}
\section{The symmetry $\rho_a$}

If $\psi\in L^\infty(\TT)$, as is usual notation, let $T_\psi\in\b(H^2)$ be the Toeplitz operator with symbol $\psi$: $T_\psi f=P_{H^2}(\psi f)$.

 The following remark is certainly well known:
\begin{lem}\label{toeplitz 1}
For $a\in\DD$, 
$$
W_a S W_a=T_{\varphi_a}=M_{\varphi_a}.
$$
\end{lem}
\begin{proof}
Straightforward computation:
$$
W_aSW_af(z)=\sqrt{1-|a|^2}\ W_a(\frac{z}{1-\bar{a}z}f(\frac{a-z}{1-\bar{a}z}))=\frac{a-z}{1-\bar{a}z}f(z).
$$
\end{proof}
Therefore:
\begin{teo}\label{toeplitz}
$$
|Ca|=\sqrt{1-|a|^2}\Big(T_{|1-\bar{a}z|^{-2}}\Big)^{1/2}=|T_{\psi_a}|.
$$
\end{teo}
\begin{proof}
As remarked above, 
$$
|C_a|^2=C_a^*C_a=\frac{1}{1-|a|^2} W_a(I-a S^*)(I-\bar{a}S)W_a=\frac{1}{1-|a|^2} W_a(I-a S^*)W_aW_a(I-\bar{a}S)W_a
$$
which by Lemma \ref{toeplitz 1} equals
$$
\frac{1}{1-|a|^2}(I-a (W_aSW_a)^*)(I-\bar{a}W_aSW_a)=\frac{1}{1-|a|^2}(I-aT_{\varphi_a}^*)(I-\bar{a}T_{\varphi_a})=\frac{1}{1-|a|^2}T_{1-\bar{a}\varphi_a}^*T_{1-\bar{a}\varphi_a}.
$$
Now we use the fact that $T_\psi^*=T_{\bar{\psi}}$ and that if $\psi,\bar{\theta}\in H^\infty$, then $T_\theta T_\psi=T_{\theta\psi}$ (see chapter 7 of Douglas' book \cite{douglas}, specifically Prop. 7.5 for the second assertion). Then
$$
C_a^*C_a=\frac{1}{1-|a|^2}T_{(1-a\bar{\varphi}_a)(1-\bar{a}\varphi_a)}.
$$
Since $1-\bar{a}\varphi_a(z)=\frac{1-|a|^2}{1-\bar{a}z}$, it follows that this expression above equals
$$
(1-|a|^2)\displaystyle{T_{\frac{1}{|1-\bar{a}z|^2}}},
$$
and the proof follows.
\end{proof}

As a consequence, we may use the remarkable description of the spectral decomposition of selfadjoint Toeplitz operators obtained by M. Rosenblum un \cite{rosenblum}. Let us quote in the next remark this description:
\begin{rem}
Suppose that $\omega:\TT\to\mathbb{R}$ is a measurable function that satifies the following conditions:
\begin{enumerate}
\item
$\omega$ is bounded from below: $\omega(\theta)>-\infty$.
\item
For each $\lambda\in\mathbb{R}$, the set 
$$
\Gamma_\lambda:=\{e^{i \theta}\in\TT: \omega(\theta)\ge \lambda\}
$$ 
is a.e. an arc.
\end{enumerate}

Then Rosenblum's {\bf Theorem 3} in \cite{rosenblum} states that:
if $T_\omega$ is the Toeplitz operator with symbol $\omega$, $\Lambda\subset\mathbb{R}$ is a Borel subset and $E(\Lambda)$ is the spectral measure (of $T_\omega$) associated to $\Lambda$, $u,v\in\DD$, one has that
\begin{equation}\label{medida espectral}
\langle E(\Lambda)k_u,k_v\rangle=\int_\Lambda \Phi(\bar{u};\lambda) \overline{\Phi(\bar{v};\lambda)} dm(\lambda),
\end{equation}
where 
$$
\Phi(u;\lambda)=\Psi(u;\lambda) \left(1-ue^{i \alpha(\lambda)}\right)^{-1/2}
\left(1-ue^{i \beta(\lambda)}\right)^{-1/2},
$$
$$
\Psi(u;\lambda)=\exp\Big(-\int_{-\pi}^\pi \log |\omega(\theta)-\lambda| P(u,\theta) d\theta \Big),
$$
$$
P(u,\theta)=\frac{1}{4\pi} \frac{1+ue^{i\theta}}{1-ue^{i\theta}},
$$
$\alpha(\lambda)\le \beta(\lambda)\in[-\pi,\pi]$ are such that
$$
\Gamma_\lambda=\{e^{i\theta}: \alpha(\lambda)\le \theta\le \beta(\lambda)\},
$$
and 
$$
dm(\lambda)=\frac{1}{\pi} \sin(\frac12(\beta(\lambda)-\alpha(\lambda))) d\lambda.
$$
In particular, note that the spectral measure of $T_\omega$ is absolutely continuous with respect to the Lebesgue measure. 
\end{rem} 
In our case, we must analyze $\omega(\theta)=\frac{1}{|1-\bar{a}e^{i\theta}|^2}=|k_a(e^{i\theta})|^2$. We consider the case $a\ne 0$ (for $a=0$ recall that $\rho_0=C_0$).  The function $\omega$ is continuous, so condition 1. is fulfilled. With respect to condition 2., note that, for $\lambda\le 0$, $\Gamma_\lambda$ is empty, and for $\lambda>0$ 
$$
\Gamma_\lambda=\{e^{i\theta}: |\frac{a}{|a|^2}-e^{i\theta}|\le \frac{1}{|a|\sqrt{\lambda}}\}.
$$
Consider the following figure

	\begin{center}
	\includegraphics[width=.6\textwidth]{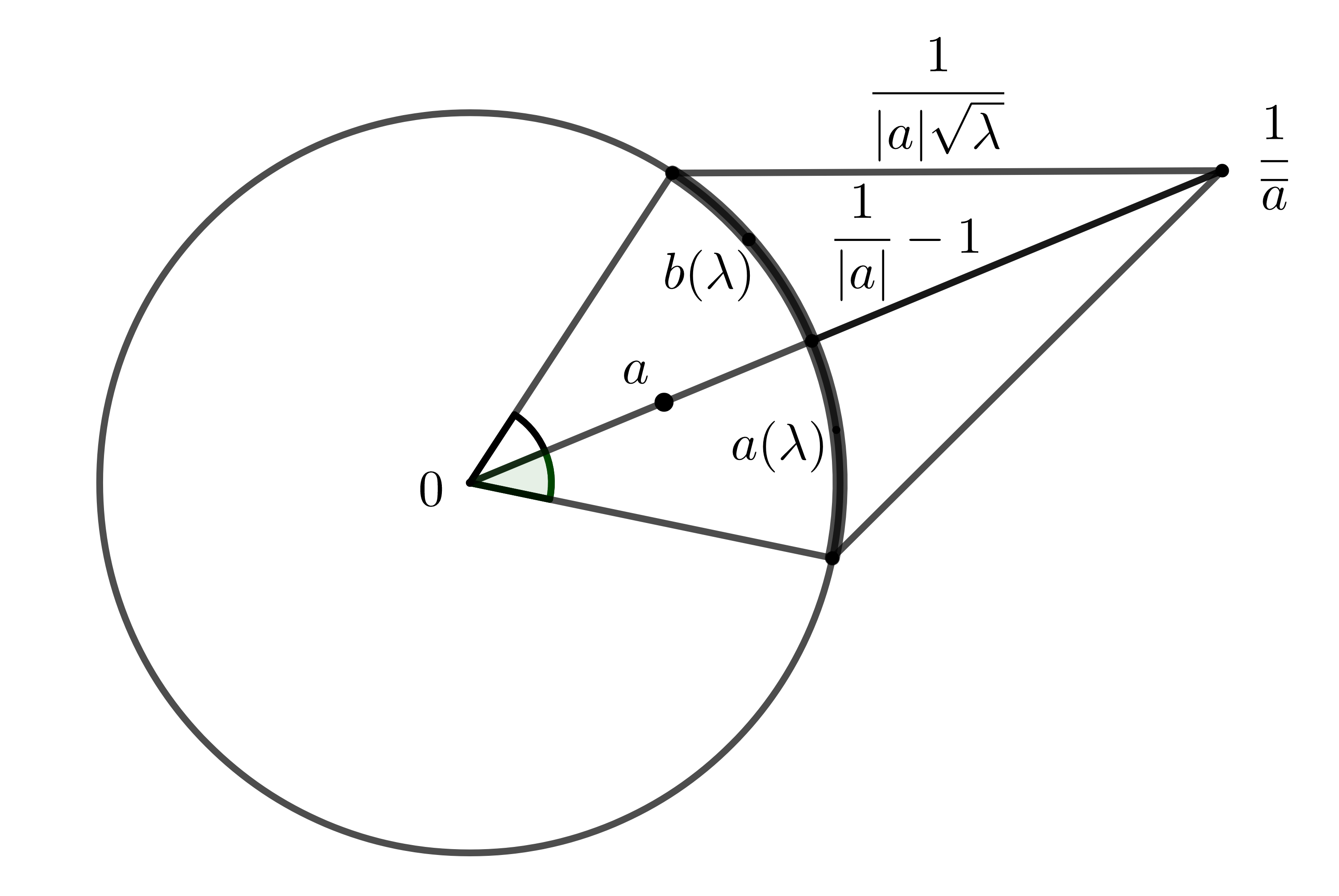}
	\end{center}
\centerline{ Figure 1}
\bigskip

Then clearly the spectral measure is zero if
\begin{itemize}
\item
$\displaystyle{\lambda > \frac{1}{(1-|a|)^2}}$ (here $\alpha(\lambda)=\beta(\lambda)$ and $\Gamma_\lambda$ has measure zero), or if
\item
$\displaystyle{\lambda < \frac{1}{(1+|a|)^2}}$ (here $\alpha(\lambda)=-\pi$, $\beta(\lambda)=\pi$ and $\Gamma_\lambda=\TT$).
\end{itemize}
 For $\lambda\in\left[ \displaystyle{\frac{1}{(1+|a|)^2}, \frac{1}{(1-|a|)^2}}\right]$ we have the following figure:

	\begin{center}
	\includegraphics[width=.6\textwidth]{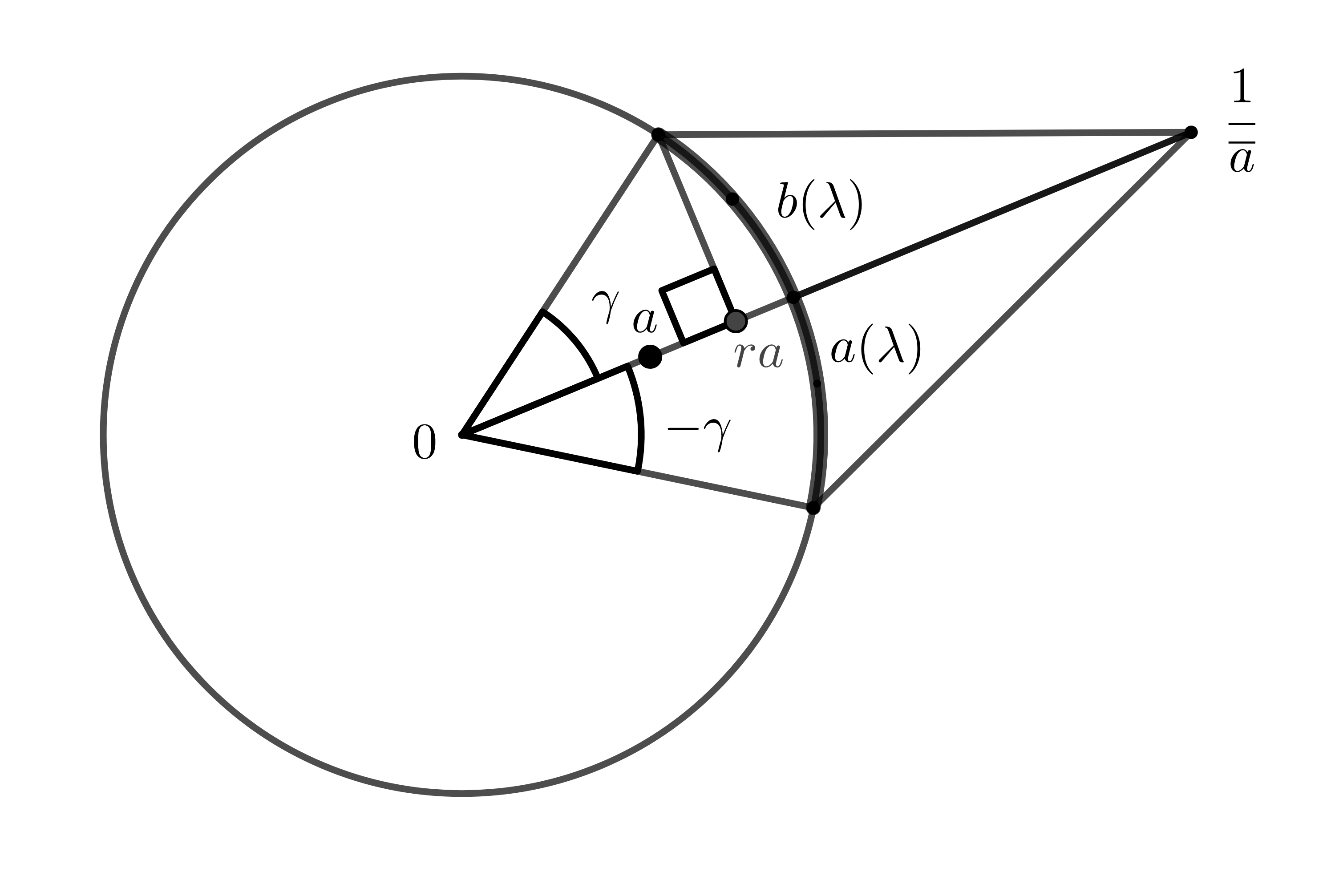}
	\end{center}
\centerline{ Figure 2}
\bigskip

Therefore, after elementary computations, one has that $\beta(\lambda)=\arcsin(\gamma)$, $\alpha(\lambda)=-\arcsin(\gamma)$ and 
$$
\sin\left(\frac12(\beta(\lambda)-\alpha(\lambda))\right)=\sin(\gamma)=\sqrt{1-\frac14\left(1+\frac{1}{|a|}(1-\frac{1}{\lambda})\right)^2}.
$$

Thus, we may characterize the function $\rho_a1$ (the symmetry $\rho_a$ at the element $1\in H^2$). To this effect, recall that the set $\{k_u:u\in\DD\}$ is total in $H^2$.
\begin{prop}
With the current notations, for $v\in\DD$, we have
$$
\langle\rho_a1,k_{v}\rangle=\frac{\sqrt{1-|a|^1}}{\pi} \displaystyle{ \int_{\frac{1}{(1+|a|)^2}}^{ \frac{1}{(1-|a|)^2}}} \lambda^{1/2} \Phi(0;\lambda) \overline{\Phi(\bar{v};\lambda)} \sqrt{1-\frac14\left(1+\frac{1}{|a|}(1-\frac{1}{\lambda})\right)^2} d\lambda.
$$
\end{prop}
\begin{proof}
Recall that 
$$
\rho_a=C_a(C_a^*C_a)^{-1/2}=(C_aC_a^*)^{-1/2}C_a=(C_a^*C_a)^{1/2}C_a,
$$
so that (since $1=k_0$)
$$
\rho_a1=|C_a|^{1/2}C_a1=|C_a|^{1/2}1=|C_a|^{1/2}k_0,
$$
and then
$$
\langle\rho_a1,k_{v}\rangle=\sqrt{1-|a|^2}\langle T_{|1-\bar{a}e^{i\theta}|^{-2}}^{1/2}k_0,k_{v}\rangle,
$$
and the formula folllows applying Rosenblum's result and the above elementary computations.
\end{proof}
\begin{rem}
In particular, we have that
$$
\rho_a1(0)=\langle\rho_a1,1\rangle=\frac{\sqrt{1-|a|^2}}{\pi}\int_{\frac{1}{(1+|a|)^2}}^{ \frac{1}{(1-|a|)^2}} \lambda^{1/2}|\Phi(0,\lambda)|^2 \sqrt{1-\frac14\left(1+\frac{1}{|a|}(1-\frac{1}{\lambda})\right)^2} d\lambda,
$$
with
$$
|\Phi(0,\lambda)|^2=\exp\left(-\frac{1}{2\pi}\int_{-\pi}^\pi  \log | |1-\bar{a}e^{i\theta}|^{-2} -\lambda| d\theta\right).
$$
\end{rem}

Clearly, if $A\subset\DD$ is a finite set, then $\{k_a: a\in\DD\setminus A\}$ is also total in $H^2$. Therefore we may characterize $\rho_a$ as follows:
\begin{teo}
With the current notations, for $a, u, v\in\DD$, with $u\ne a$, we have
$$
\langle\rho_ak_u,k_{v}\rangle=\frac{\bar{u}(|a|^2-1)^{3/2}}{\pi(\bar{u}-\bar{a})}\int_{\frac{1}{(1+|a|)^2}}^{\frac{1}{(1-|a|)^2}} \lambda^{1/2} \Phi(\varphi_a(u);\lambda) \overline{\Phi(\bar{v};\lambda)} \sqrt{1-\frac14\left(1+\frac{1}{|a|}(1-\frac{1}{\lambda})\right)^2} d\lambda +
$$
$$
+\frac{\bar{a}}{\bar{a}-\bar{u}}\frac{\sqrt{1-|a|^2}}{\pi}\int_{\frac{1}{(1+|a|)^2}}^{\frac{1}{(1-|a|)^2}} \lambda^{1/2} \Phi(0;\lambda) \overline{\Phi(\bar{v};\lambda)} \sqrt{1-\frac14\left(1+\frac{1}{|a|}(1-\frac{1}{\lambda})\right)^2} d\lambda.
$$
These inner products characterize $\rho_a$, because $\{k_u: u\in\DD, u\ne a\}$ is a total set in $H^2$.
\end{teo}
\begin{proof}
The last assertion is clear.

Recall that 
$$
\rho_a=C_a(C_a^*C_a)^{-1/2}=(C_aC_a^*)^{-1/2}C_a=(C_a^*C_a)^{1/2}C_a.
$$
Note that 
$$
C_ak_u(z)=\frac{1-\bar{a}z}{1-\bar{u}a-z(\bar{a}-\bar{u})}=\frac{1}{1-\bar{u}a}\frac{1-\bar{a}z}{1-\overline{\varphi_a(u)}z},
$$
which after routine computations (using that $a\ne u$, and $1=k_0$) yields
$$
C_ak_u=\frac{\bar{u}(1-|a|^2)}{\bar{u}-\bar{a}} k_{\overline{\varphi_a(u)}}+\frac{\bar{a}}{\bar{a}-\bar{u}} k_0.
$$
Therefore, 
$$
\rho_a k_u=(C_a^*C_a)^{1/2}C_ak_u=(C_a^*C_a)^{1/2}(\frac{\bar{u}(1-|a|^2)}{\bar{u}-\bar{a}} k_{\overline{\varphi_a(u)}}+\frac{\bar{a}}{\bar{a}-\bar{u}} k_0),
$$
and thus
$$
\langle \rho_ak_u,k_v\rangle=\sqrt{1-|a|^2}\langle T_{|1-\bar{a}e^{i\theta}|^{-2}}^{1/2}(\frac{\bar{u}(1-|a|^2)}{\bar{u}-\bar{a}} k_{\overline{\varphi_a(u)}}+\frac{\bar{a}}{\bar{a}-\bar{u}} k_0),k_{v}\rangle
$$
$$
=\sqrt{1-|a|^2}\{\frac{\bar{u}(1-|a|^2)}{\bar{u}-\bar{a}}\langle T_{|1-\bar{a}e^{i\theta}|^{-2}}^{1/2}k_{\overline{\varphi_a(u)}},k_v\rangle+\frac{\bar{a}}{\bar{a}-\bar{u}}\langle T_{|1-\bar{a}e^{i\theta}|^{-2}}^{1/2}k_0,k_v\rangle\}.
$$
The formula folllows applying Rosenblum's result and the above elementary computations.
\end{proof}

\subsection{A result by E. Berkson}

We are indebted to Daniel Su\'arez for pointing us the result below.
In \cite{berkson}, E. Berkson proved the following Theorem:

\begin{teo}{\rm \cite{berkson}}
Let $\varphi:\DD\to\DD$ be a bounded analytic map,  $\tilde{\varphi}$ its boundary function, and $A=\tilde{\varphi}^{-1}(\TT)$. Suppose that $|A|>0$ ($=$ normalized Lebesgue measure in $\TT$). Let $\psi:\DD\to\DD$ be another analytic map,  and $C_\varphi$ and $C_\psi$ denote the composition operators on $H^p(\DD)$, $1\le p<\infty$.
If 
$$
\|C_\psi-C_\varphi\|<\left(\frac{|A|}{2}\right)^{1/p},
$$
then $\psi=\varphi$.
\end{teo}
As a consequence, for $a\ne b\in\DD$ we have that ($p=2$):
\begin{equation}\label{berkson}
\|C_a-C_b\|\ge \frac{1}{\sqrt{2}}
\end{equation}

On the other hand, it is a consequence of Theorem \ref{toeplitz} that
$$
C_a^*C_a-C_b^*C_b=\displaystyle{T_{\frac{1-|a|^2}{|1-\bar{a}z|^2}}-T_{\frac{1-|b|^2}{|1-\bar{b}z|^2}}=T_{\delta_{a,b}}},
$$
where $\delta_{a,b}(z)=\displaystyle{\frac{1-|a|^2}{|1-\bar{a}z|^2}-\frac{1-|b|^2}{|1-\bar{b}z|^2}}$. Thus
$$
\|C^*_aC_a-C_b^*C_b\|=\|\delta_{a,b}\|_\infty=\sup\{|\delta_{a,b}(z)|: z\in\TT\}.
$$
In particular, contrary to what happens to $C_b$ and $C_a$, if $b\to a$, then both $C_b^*C_b\to C_a^*C_a$ and $|C_b|\to |C_a|$.
Therefore, we have the following:
\begin{prop}\label{coro berkson}
Fix $a\in\DD$ and $r<\frac{1}{\sqrt{2}}$, consider the open neighbourhood $\b_r(a)$ of $a$ in $\DD$ given by
$$
\b_r(a):=\{b\in\DD: \||C_b|-|C_a|\|<r\}.
$$
Then, if $b\in\b_r(a)$, $b\ne a$, we have that 
$$
\|\rho_b-\rho_a\|\ge \left(\frac{1}{\sqrt{2}}-r\right) \frac{1+|a|}{\sqrt{1-|a|^2}}.
$$
\end{prop}
\begin{proof}
By Berkson's Theorem, if $a\ne b$
$$
\frac{1}{\sqrt{2}}\le\|C_a-C_b\|=\|\rho_a|C_a|-\rho_b|C_b|\|\le \|\rho_a|C_a|-\rho_b|C_a|\|+\|\rho_b|C_a|-\rho_b|C_b|\|
$$
$$
\le \||C_a|\|\|\rho_a-\rho_b\|+\||C_a|-|C_b|\|,
$$
because $\rho_b$ is a unitary operator. If $b\in\b_r(a)$,
$$
\frac{1}{\sqrt{2}}\le \|C_a\|\|\rho_a-\rho_b\|+r.
$$
The proof follows recalling that $\||C_a|\|=\|C_a\|=\frac{\sqrt{1-|a|^2}}{1+|a|}$. 
\end{proof}

\subsection{Formulas for $\rr(C_a)$ and $\nn(C_a)$.} 

Using Theorem \ref{toeplitz} we can refine the formulas for $\rr(T)$ and $\nn(T)$ obtained in Proposition \ref{r(T) y n(T)},  the range and nullspace symmetries induced by a reflection $T$, to the case when $T=C_a$:
\begin{coro}
We have
$$
\rr(C_a)=2(I+C_a)T_{\gg_a}^{-1} \ \hbox{ and } \ \ \nn(C_a)=2(I-C_a)T_{\gg_a}^{-1},
$$
where $T_{\gg_a}$ is the Toeplitz operator with symbol
$$
\gg_a(z)=1+\frac{1-|a|^2}{|1-\bar{a}z|^2}.
$$
\end{coro}
\begin{proof}
Note  that for $T=C_a$ we have  $\nn(C_a)=2(I+C_a)(|C_a|^2+I)^{-1}$, and from Theorem \ref{toeplitz} we know that 
$$
|C_a|^2=(1-|a|^2)\displaystyle{T_{\frac{1}{|1-\bar{a}z|^2}}}.
$$
Then 
$$
|C_a|^2+I=(1-|a|^2)\displaystyle{T_{\frac{1}{|1-\bar{a}z|^2}}}+I=\displaystyle{T_{1+\frac{1-|a|^2}{|1-\bar{a}z|^2}}}=T_{\gg_a}.
$$
The computation of $\nn(C_a)$ is similar.

\end{proof}
\subsection{A power series expansion for $\rho_a$}

Let us further consider $|I-\bar{a}S|^{-1}$. Note that 
$$
|I-\bar{a}S|^2=(I-aS^*)(I-\bar{a}S)=I+|a|^2-2 {\rm Re}(\bar{a}S),
$$
where ${\rm Re}T=\frac12 (T+T^*)$, for $T\in\b(H^2)$, as is usual notation. Then
$$
|I-\bar{a}S|^2=(1+|a|^2)\left( I-\frac{2}{1+|a|^2}{\rm Re}(\bar{a}S)\right)=(1+|a|^2)\left( I-\frac{2|a|}{1+|a|^2} T\right),
$$ 
where $a=|a|\omega$ and  $T={\rm Re}( \bar{\omega}S)$ is a contraction. Using the power series expansion $(1-kt)^{-1/2}=1+\sum_{n=1}^\infty (2n-1)(2n-3)\dots 1 (\frac{k}{2})^n t^n$, we get
\begin{lem}
With the current notations, i.e. $T={\rm Re}(\bar{\omega}S)$, $a=|a|\omega$, we have that
\begin{enumerate}
\item
$$
|I-\bar{a}S|^{-1}=\frac{1}{\sqrt{1+|a|^2}}\left(I+ \sum_{n=1
}^\infty (2n-1)(2n-3)\dots 1 (\frac{|a|}{1+|a|^2})^n T^n\right),
$$
where $T=\frac12(\bar{\omega}S +\omega S^*)$.
\item
$$
|C_a|^{-1}=\sqrt{1-|a|^2}W_a\{I+ \sum_{n=1
}^\infty (2n-1)(2n-3)\dots 1 (\frac{|a|}{1+|a|^2})^n T^n\}W_a
$$
$$
=\sqrt{1-|a|^2}\left(I+ \sum_{n=1
}^\infty (2n-1)(2n-3)\dots 1 (\frac{|a|}{1+|a|^2})^n (W_aTW_a)^n\right).
$$
\item
$$\rho_a=(1-\bar{a}S)\{I+ \sum_{n=1
}^\infty (2n-1)(2n-3)\dots 1 (\frac{|a|}{1+|a|^2})^n T^n\}W_a=\Big(\mu(I-\bar{a}S)\Big)W_a,
$$
where $\mu(A)=$ unitary part in the polar decomposition of $A$: $A=\mu(A)|A|$.
\end{enumerate}
\end{lem}
\begin{proof}
Straightforward computations.
\end{proof}
Next we see that the map $\DD\ni a \mapsto |C_a|$ is one to one:
\begin{prop}
Let $a,b\in\DD$. Then $|C_a|=|C_b|$ if and only if $|C_a^*|=|C_b^*|$ if and only if $a=b$
\end{prop}
\begin{proof}
Recall that $(C_a^*C_a)^{-1}=C_aC_a^*$, and thus $|C_a|^{-1}=|C_a^*|$. By uniqueness of the positive square root of operators, clearly $|C_a^*|=|C_b^*|$ if and only if $C_aC_a^*=C_bC_b^*$. Next note that at the constant function $1\in H^2$, we have (since $S^*1=0$)
$$
C_aC_a^*(1)=\frac{1}{1-|a|^2}(I-\bar{a}S)(I-aS^*)(1)=\frac{1}{1-|a|^2}(I-\bar{a}S)(1)=\frac{1-\bar{a}z}{1-|a|^2}.
$$
 Evaluating at $z=0$, we get that $C_aC_a^*=C_bC_b^*$ implies that $|a|=|b|$, and thus $1-\bar{a}z=1-\bar{b}z$ for all $z\in\DD$, i.e., $a=b$.
\end{proof}
\begin{quest}
Proposition \ref{coro berkson} states that given $a\in\DD$, there is an open neighbourhood of $a$ such that for $b$ in this neighbourhood, $\rho_a=\rho_b$ implies $a=b$. We do not now though if globally the map $\DD\ni a\mapsto\rho_a\in\b(H^2)$ is injective.
\end{quest}
\section{The eigenspaces of $C_a$}

Denote by $\e$ and $\o$ the closed subspaces of even and odd functions in $H^2$. Note that they are,  respectively,  $\e=N(C_0-I)$ and $\o=N(C_0+I)$. For general $a\in\DD$, the eigenspaces of $C_a$  are 
$$
N(C_a-I)=\{f\in H^2: f\circ \varphi_a=f\}\ \ \hbox{ and } \ \ N(C_a+I)=\{g\in H^2: g\circ \varphi_a=-g\}.
$$

For $a\in\DD$, recall from (\ref{punto fijo omega}) the fixed point $\omega_a$  of $\varphi_a$ inside $\DD$. 
Elementary computations show that 
\begin{equation}\label{composicion punto fijo 1}
\varphi_{\omega_a}\circ\varphi_a=-\varphi_{\omega_a}
\end{equation}
which at $z=0$ gives
\begin{equation}\label{composicion punto fijo 2}
\varphi_{\omega_a}(a)=-\omega_a.
\end{equation}
\begin{teo}
For $a\in\DD$, the eigenspaces of $C_a$ are
\begin{equation}\label{a pares}
N(C_a-I)=\{f=\sum_{n=0}^\infty \alpha_n(\varphi_{\omega_a})^{2n}: (\alpha_n)\in\ell^2\}=C_{\omega_a}(\e),
\end{equation}
and
\begin{equation}\label{a impares}
N(C_a+I)=\{g=\sum_{n=0}^\infty \alpha_n(\varphi_{\omega_a})^{2n+1}: (\alpha_n)\in\ell^2\}=C_{\omega_a}(\o).
\end{equation}
\end{teo}
\begin{proof}
It follows from (\ref{composicion punto fijo 1}) that the {\bf even} powers of $\varphi_{\omega_a}$ belong to $N(C_a-I)$:
$$
(\varphi_{\omega_a})^{2n}\circ\varphi_a=(\varphi_{\omega_a})^{2n},
$$
and the {\bf odd} powers belong to $N(C_a+I)$:
$$
(\varphi_{\omega_a})^{2n+1}\circ\varphi_a=-(\varphi_{\omega_a})^{2n+1}.
$$
Therefore, any sequence of coefficients $(\alpha_n)\in\ell^2$ gives an element
$$
f=\sum_{n=0}^\infty \alpha_n(\varphi_{\omega_a})^{2n}\in N(C_a-I),
$$
and an element
$$
g=\sum_{n=0}^\infty \alpha_n(\varphi_{\omega_a})^{2n+1}\in N(C_a+I).
$$
Conversely, suppose that $f\in N(C_a-I)$. 
Using (\ref{composicion punto fijo 2})
$$
f\circ\varphi_{\omega_a}=f\circ\varphi_a\circ\varphi_{\omega_a},
$$
and since $\varphi_a\circ\varphi_{\omega_a}=\frac{a\bar{\omega}_a-1}{1-\bar{a}\omega_a} \varphi_{\varphi_{\omega_a}(a)}=-\varphi_{-\omega_a}$, we get 
$$
f\circ\varphi_{\omega_a}(z)=f\circ\varphi_{\omega_a}(-z),
$$
i.e., $f\circ\varphi_{\omega_a}\in\e$. The fact for odd functions is similar.
\end{proof}
Note that if we denote $h(z)=\sum_{n=0}^\infty \alpha_n z^{2n}$, which is an arbitrary even function in $H^2$, we have that $f=h\circ\varphi_{\omega_a}=C_{\omega_a}h$.  And similarly if $k(z)=\sum_{n=0}^\infty \alpha_n z^{2n+1}$ is an arbitrary odd function in $H^2$, $g=C_{\omega_a}k$.
Then
$$
C_{\omega_a}\big|_\e:\e\to N(C_a-I) \ \hbox{ and } \ C_{\omega_a}\big|_\o:\o\to N(C_a+I).
$$  
\begin{teo}
The restrictions $C_{\omega_a}\big|_\e$ and $C_{\omega_a}\big|_\o$ are bounded linear isomorphisms. Their inverses are, respectively, $C_{\omega_a}\big|_{N(C_a-I)}$ and  $C_{\omega_a}\big|_{N(C_a+I)}$.
\end{teo}
\begin{proof}
Note that 
$$
H^2=C_{\omega_a}(\e\oplus\o)=C_{\omega_a}(\e)\dot{+}C_{\omega_a}(\o)\subset N(C_a-I)\dot{+}N(C_a+I),
$$
were $\dot{+}$ denotes direct (non necessarily orthogonal)  sum. It follows that 
$C_{\omega_a}(\e)=N(C_a-I)$ and $C_{\omega_a}(\o)=N(C_a+I)$. This completes the proof, since $C_a$ is its own inverse.
\end{proof}

\begin{rem}\label{inner outer}
Clearly, if $p,g \in H^2$ are, respectively, inner and outer  functions, then $C_ap=p\circ\varphi_a$ and $C_ag=g\circ\varphi_a$ are also, respectively, inner and outer. Therefore, if $f\in N(C_a-I)$, and $f=pg$ is the inner/outer factorization of $f$, then $f=C_ap\cdot C_ag$ is another inner/outer factorization. By uniqueness, it must be $C_ap=\omega p$ for some $\omega\in\mathbb{T}$. But then $p$ is an eigenfunction of $C_a$, and so it must be $\omega=\pm 1$. Therefore, if $f \in N(C_a-I)$, then either {\bf a)} $p,g\in N(C_a-I)$ or {\bf b)} $p,g\in N(C_a+I)$. The latter case cannot happen: the outer function $g$ verifies that $C_{\omega_a}g$ is odd, and therefore it vanishes at $z=0$,
$$
0=C_{\omega_a}g(0)=g(\omega_a).
$$
A similar consideration can be done for $N(C_a+I)$.  If $f=pg$ is the inner/outer factorization of $f\in N(C_a+I)$, then again $C_ap=\pm p$. If $C_ap=p$, then 
$$
-f=-pg=f\circ\varphi_a=(p\circ\varphi_a)(g\circ\varphi_a)
$$
implies $p\circ\varphi_a=\pm p$. If $p\circ\varphi_a=p$, then $g\circ\varphi_a=-g$, and therefore the outer function $g$ vanishes, a contradiction. Thus $p\in N(C_a+I)$ and $g\in N(C_a-I)$.
\end{rem}

Let us examine the position of the subspaces $N(C_a\pm I)$ and their orthogonal complements. To this effect, the following result will be needed:
\begin{lem}\label{el lema}
Let $a\ne b\in\mathbb{D}$. If $f\in H^1$ satisfies that $f\circ\varphi_a=f=f\circ\varphi_b$, then $f$ is constant.
\end{lem}
\begin{proof}
We know that  
\begin{equation}\label{otra vez}
\varphi_{\omega_a}\varphi_a\varphi_{\omega_a}=\varphi_0.
\end{equation}
An elementary computations shows that, for any $c\in\DD$, $f\circ\varphi_c=f$ if and only if $h=f\circ\varphi_{\omega_a}$ satisfies
\begin{equation}\label{truco c}
h\circ (\varphi_{\omega_a}\circ\varphi_c\circ\varphi_{\omega_a})=h
\end{equation}
If we use (\ref{truco c}) for $c=a$, we get, in view of (\ref{otra vez}), that
$$
h\circ\varphi_0=h.
$$
Another straigtforward computation shows that  for $b,d\in\DD$
\begin{equation}\label{conjugacion general}
\varphi_d\circ\varphi_b\circ\varphi_d=\varphi_{d\bullet b} , \hbox{ where } \ d\bullet b:=\frac{2d-b-\bar{b}d^2}{1+|d|^2-\bar{b}d-b\bar{d}}.
\end{equation}
Then, if we apply (\ref{truco c}) for $c=b$, we get that $h$ satisfies
$$
h\circ\varphi_{\omega_a\bullet b}=h.
$$
Clearly $h$ is constant if and only if $f$ is constant. Thus, we have reduced to the case when one of the to points is the origin:
$$
f=f\circ\varphi_0=f\circ\varphi_a.
$$
In particular, this implies that
$$
f=f\circ\varphi_a\circ\varphi_0\circ\dots\circ\varphi_a=f\circ(\varphi_a\circ\varphi_0)^{(n)}\circ\varphi_a,
$$
for all $n\ge 1$ (here $(\varphi_a\circ\varphi_0)^{(n)}$ denotes the composition of $\varphi_a\circ\varphi_0$ with itself $n$ times). We shall need the following computation:
\begin{claim}\label{lema 55}
$$
(\varphi_a\circ\varphi_0)^{(n)}\varphi_a=\varphi_{a_n},
$$
where 
$$
a_n=\displaystyle{\frac{a}{|a|}\ \frac{1-\left(\frac{1-|a|}{1+|a|}\right)^{n+1}}{1+\left(\frac{1-|a|}{1+|a|}\right)^{n+1}}}.
$$
\end{claim}
\begin{proof}
Our claim is equivalent to
$$
a_n=\frac{a}{|a|} \ \frac{(1+|a|)^{n+1}-(1-|a|)^{n+1}}{(1+|a|)^{n+1}+(1-|a|)^{n+1}}.
$$
The proof is by induction in $n$. It is an elementary computation. For $n=1$, we have  that 
$$
\varphi_a\circ\varphi_0\circ\varphi_a(z)=\varphi_a(-\frac{a-z}{1-\bar{a}z})=\frac{a+\frac{a-z}{1-\bar{a}z}}{1+\bar{a}\frac{a-z}{1-\bar{a}z}}=\frac{2a-(1+|a|^2)z}{1+|a|^2-2az}=\frac{\frac{2a}{1+|a|^2}-z}{1-\frac{2\bar{a}}{1+|a|^2}z}=\varphi_{\frac{2a}{1+|a|^1}}(z).
$$
On the other hand, 
$$
a_1=\frac{a}{|a|}\ \frac{(1+|a|)^2-(1+|a|)^2}{(1+|a|)^{2}+(1-|a|)^{2}}=\frac{2a}{1+|a|^2}.
$$
Suppose the formula valid for $n$.
Then 
$$
(\varphi\circ\varphi_0)^{n+1}\circ\varphi_a(z)=(\varphi\circ\varphi_0)^{n}\circ\varphi_a\circ(\varphi_a\circ\varphi_0)(z)=\varphi_{a_n}(-\frac{a-z}{1-\bar{a}z})
$$
$$
=\displaystyle{\frac{a_n+\frac{a-z}{1-\bar{a}z}}{1+\bar{a}_n\frac{a-z}{1-\bar{a}z}}}= \displaystyle{
\frac{\frac{a}{|a|} {\bf f}_n+\frac{a-z}{1-\bar{a}z}}{1-\frac{a}{|a|} {\bf f}_n \frac{a-z}{1-\bar{a}z}}}=\frac{a(\frac{{\bf f}_n}{|a|}+1)-(|a|{\bf f}_n+1)z}{|a|{\bf f}_n+1-\bar{a}(\frac{{\bf f}_n}{|a|}+1)z}=\frac{\beta_n-z}{1-\bar{\beta}_nz}=\varphi_{\beta_n}(z),
$$
where 
$$
\beta_n=a\ \frac{(\frac{{\bf f}_n}{|a|}+1)}{|a|{\bf f}_n+1} \ \hbox{ and } \ {\bf f}_n=\frac{(1+|a|)^{n+1}-(1-|a|)^{n+1}}{(1+|a|)^{n+1}+(1-|a|)^{n+1}}.
$$
Thus, we have to show that $\beta_n=a_n$. Note that
$$
\beta_n=\frac{a}{|a|}\ \frac{{\bf f}_n+|a|}{|a|{\bf f}_n+1}
$$
and that
$$
\frac{{\bf f}_n+|a|}{|a|{\bf f}_n+1}=\frac{(1+|a|)^{n+1}-(1-|a|)^{n+1}+|a|(1-|a|)^{n+1}+|a|(1+|a|)^{n+1}}{-|a|(1-|a|)^{n+1}+(1+|a|)^{n+1}+(1-|a|)^{n+1}+|a|(1+|a|)^{n+1}}
$$
$$
=\frac{(1+|a|)^{n+2}-(1-|a|)^{n+2}}{(1+|a|)^{n+2}+(1-|a|)^{n+2}},
$$
which completes the proof of  Claim \ref{lema 55}.
\end{proof}
Returning to the proof of the Lemma, suppose that there is a non constant $f$ such that $\circ\varphi_0=f=f\circ\varphi_a$.  Then $f_0=f-f(0)$ has the same property. As remarked above, $f_0=f_0\circ\varphi_{a_n}$ for all $n\ge 0$ (for $n=0$, $a_0=a$). It follows that $0$ and $a_n$, $n\ge 0$ are zeros of $f_0$. Since $f_0$ is also even, also $-a_n$, $\ge 0$ occur as zeros of $f_0$.  Consider   $f_0=BSg$ the factorization of $f_0$ with $B$ a Blaschke product, $S$ singular inner and $g$ outer.  Then the pairs of factors
$$
\varphi_{a_n}\cdot \varphi_{-a_n}
$$
appear in the expression of $B$. Since $f_0=f_0\circ\varphi_a$, and $S\circ \varphi_a$ and $g\circ\varphi_a$ are non vanishing in $\DD$, it follows that
$$
(\varphi_{a_n}\circ\varphi_a)\cdot (\varphi_{-a_n}\circ\varphi_a)
$$
must also appear in the expression of $B$. Note that
$$
\varphi_{a_n}\circ\varphi_a=(\varphi_a\circ\varphi_0)^{(n)}\circ\varphi_a\circ\varphi_a=((\varphi_a\circ\varphi_0)^{(n)}=(\varphi_a\circ\varphi_0)^{(n-1)}\circ\varphi_a\circ\varphi_0=\varphi_{a_{n-1}}\circ\varphi_0.
$$
Also
$$
\varphi_{-a_n}(z)=-\frac{a_n+z}{1+\bar{a}_nz}=-\varphi_{a_n}(-z)=\varphi_0\circ\varphi_{a_n}\circ\varphi_0=\varphi_0\circ(\varphi_a\circ\varphi_0)^{(n)}\circ\varphi_a\circ\varphi_0=\varphi_0\circ(\varphi_a\circ\varphi_0)^{(n+1)}.
$$
Then
$$
\varphi_{-a_n}\circ\varphi_a=\varphi_0\circ(\varphi_a\circ\varphi_0)^{(n+1)}\circ\varphi_a=\varphi_0\circ\varphi_{a_{n+1}}.
$$
Note the effect of $C_a$ (i.e., of composing with $\varphi_a$, which is also defined in $H^1$) on the following pairs of factors of $B$:
$$
z\cdot z=z^2=\varphi_0^2\stackrel{C_a}{\longrightarrow}(\varphi_0\circ\varphi_a)^2=\varphi_a^2,
$$
$$
\varphi_a\cdot\varphi_{-a}\stackrel{C_a}{\longrightarrow} (\varphi_a\circ\varphi_a)\cdot(\varphi_{-a}\circ\varphi_a)=z\cdot(-\varphi_{a_1})=-z\varphi_{a_1},
$$
and 
$$
\varphi_{a_1}\cdot\varphi_{a_{-1}}\stackrel{C_a}{\longrightarrow}(\varphi_{a_1}\circ\varphi_a)\cdot(\varphi_{a_{-1}}\circ\varphi_a)=(\varphi_a\circ\varphi_0)\cdot\varphi_{a_2}=-\varphi_{-a}\cdot\varphi_{a_2}.
$$
Other pairs of factors in the expression of $B$, after applying $C_a$, do not involve $\varphi_a$ or $\varphi_0$, due to the spreading of the indices. Summarizing, after applying $C_a$, we get the products
$$
(\varphi_a)^2 , \ -z\varphi_{a_1} \hbox{ and }   -\varphi_{-a}\cdot\varphi_{a_2},
$$
i.e., we do not recover the original factors $z^2$ and $\varphi_a\cdot\varphi_{-a}$.
It follows that $f$ is constant.
\end{proof}

\begin{teo}\label{teo 54}
Let $a\ne b$ in $\DD$. Then 
\begin{enumerate}
\item
$$
N(C_a-I)\cap N(C_b-I)=\mathbb{C}1,
$$
where $1\in H^2$ is the constant function.
\item
$$
N(C_a+I)\cap N(C_b+I)=\{0\}.
$$
\end{enumerate}
\end{teo}
\begin{proof}
Assertion 1. is a particular case of Lemma \ref{el lema}. 

To prove 2., a similar trick as in the beginning of the proof of Lemma \ref{lema 55} allows us to reduce to the case of $a\ne 0$ and $b=0$, i.e., we must prove that there are no nontrivial odd functions in $N(C_a+I)$. Let $f\in H^2$ be odd such that $f\circ\varphi_a=-f$. Then $f^2=f\cdot f$  in $H^1$ is even and 
$(f(\varphi_a(z)))^2=(-f(z))^2=(f(z))^2$, i.e., $f^2\circ\varphi_a=f^2=f^2\circ\varphi_0$. Therefore, by Lemma \ref{el lema}, $f^2$ is constant, and therefore $f\equiv 0$.  
\end{proof}

\begin{coro}
The maps $\DD\to\p$ given by
$$
a\mapsto \rr(C_a) \ \hbox{ and } \ a\mapsto \nn(C_a)
$$
are one to one.
\end{coro}

Let us further proceed with the study of the position of the subspaces $N(C_a\pm I)$ and $N(C_b\pm I)$ for $a\ne b$, considering now their orthogonal complements. We shall restrict to the case $b=0$. The conditions look similar, but as far as we could figure it out, some of the proofs may be quite different.
\begin{teo}\label{0 y a}
Let $a\in\DD$, $a\ne 0$.
\begin{enumerate}
\item
$$
N(C_0-I)^\perp \cap N(C_a-I)=\{0\}=N(C_0-I)\cap N(C_a-I)^\perp,
$$
\item
$$
N(C_0+I)^\perp \cap N(C_a+I)=\{0\}=N(C_0+I)\cap N(C_a+I)^\perp,
$$
\item
$$
N(C_0-I)^\perp \cap N(C_a-I)^\perp=\{0\}=N(C_0+I)^\perp\cap N(C_a+I)^\perp.
$$
\end{enumerate}
\end{teo}
\begin{proof}
\underline{Assertion 1.}: for the left hand equality,  let $f\in N(C_0-I)^\perp=\o$ such that $f\circ\varphi_a=f$. Then, by the above results, $f^2\in\e\cap N(C_a-I)$, and therefore $f^2$ is constant. Then $f$, being constant and odd, is zero.

The right hand equality: suppose $f\in N(C_0-I)\cap N(C_a-I)^\perp$ is $\ne 0$, i.e., $f$ is even and $\langle f, C_{\omega_a}(z^{2k})\rangle=0$ for $k\ge 0$ (in particular, when $n=0$ we get $f(0)=0$). Thus  $C^*_{\omega_a}(f)$ is odd. Recall that 
$$
C^*_{\omega_a}(f)=\displaystyle{\frac{1}{1-\bar{\omega}_az} f(\frac{\omega_a-z}{1-\bar{\omega}_az})-\frac{\omega_a}{1-\bar{\omega}_az} \frac{f(\frac{\omega_a-z}{1-\bar{\omega}_az})-f(0)}{\frac{\omega_a-z}{1-\bar{\omega}_az}}}
$$
$$
=\frac{1}{1-\bar{\omega}_az}\left(f(\frac{\omega_a-z}{1-\bar{\omega}_az})-\omega_a\frac{f(\frac{\omega_a-z}{1-\bar{\omega}_az})}{\frac{\omega_a-z}{1-\bar{\omega}_az}}\right). 
$$
Since $f$ is even with $f(0)=0$, put $f(z)=\sum_{n=1}^\infty \alpha_n z^{2n}$. Then, after routine computations we get
$$
C^*_{\omega_a}(f)=z \frac{|\omega_a|^2-1}{(1-\bar{\omega}_az)^2} \sum_{n=1}^\infty \alpha_n (\frac{\omega_a-z}{1-\bar{\omega}_az})^{2n-1}.
$$
The fact that $C^*_{\omega_a}(f)$ is odd, implies that 
$$
A(z)=\frac{1}{(1-\bar{\omega}_az)^2} \sum_{n=1}^\infty \alpha_n (\frac{\omega_a-z}{1-\bar{\omega}_az})^{2n-1}
$$
is even. Note that therefore
$$
C_{\omega_a}(A)=\frac{(1-\bar{\omega}_az)^2}{(1-|\omega_a|^2)^2} \sum_{n=1}^\infty \alpha_n z^{2n-1} \in N(C_a-I).
$$
Let us abreviate $\alpha(z)=\sum_{n=1}^\infty \alpha_n z^{2n-1}$, which is an odd function. Thus 
$$
(1-\bar{\omega}_az)^2 \alpha \in N(C_a-I).
$$
Note that $(1-\bar{\omega}_az)^2$ is outer. Therefore, if $\alpha=p g$ is the inner/outer factorization of $\alpha$, then 
$$
(1-\bar{\omega}_az)^2 \alpha= p \left( (1-\bar{\omega}_az)^2g\right)
$$
is also an inner/outer factorization. Then, by Remark \ref{inner outer}, we have  $p \in N(C_a-I)$. By a similar argument, since $\alpha$ is odd it follows that $p$ is either odd or even. Note that $p$ even would imply $g$ odd, and thus vanishing at $z=0$, which cannot happen. Thus $p\in N(C_a-I)\cap \o=\{0\}$, which is the first assertion of this theorem.
Clearly this implies that $f=0$.

\underline{Assertion 2.}: the proof of the second assertion is similar. Let us sketch it underlining the differences. The left hand equality: suppose that $f$ is odd and $f\perp N(C_a+I)$. Then $f=C_{\omega_a}\iota=\iota(\varphi_{\omega_a})$ for some odd function $\iota$. Then $f^2=\iota^2(\varphi_{\omega_a})\in N(C_a-I)$.   Then, by the first part of Theorem \ref{teo 54}, we have that $f^2$ is constant, then $f$ is constant, and the fact that $f=\iota(\varphi_{\omega_a})$ with $\iota$ odd implies that $f=0$.

The right hand equality of the second assertion, if $f\in N(C_0+I)\cap N(C_a+I)^\perp$, then $f$ is odd,  $f(z)=\sum_{k\ge 0}\beta_k z^{2k+1}$ and $C_{\omega_a}^*f$ is even. Similarly as above,
$$
C_{\omega_a}^*f(z)=z\frac{|\omega_a|^2-1}{(1-\bar{\omega}_az)^2}\sum_{k\ge 0} \beta_k \left(\frac{\omega_a-z}{1-\bar{\omega}_az}\right)^{2k},
$$
and thus $B(z)=\frac{1}{(1-\bar{\omega}_az)^2} \sum_{k\ge 0}\beta_k \left(\frac{\omega_a-z}{1-\bar{\omega}_az}\right)^{2k}$ is odd. Therefore, if $\beta(z):=\sum_{k\ge 0} \beta_k z^{2k}$, we have
$$
C_{\omega_a}(B)=\frac{(1-\bar{\omega}_az)^2}{(1-\bar{\omega}_az)^2}\beta \in N(C_a+I), \ i.e., \ (1-\bar{\omega}_az)^2\beta\in N(C_a+I).
$$
If $\beta=qh$ is the inner/outer factorization, then $q$ and $h$ are even, and 
$$
(1-\bar{\omega}_az)^2\beta=q\left((1-\bar{\omega}_az)^2h\right)
$$
is the inner/outer factorization of an element in $N(C_a+I)$. Then, again by Remark \ref{inner outer}, $q\in N(C_a+I)$. Then $q^2$ is even and lies in $N(C_a-I)$, and therefore is constant, by the first part of Theorem \ref{teo 54}. Thus $q$ is constant in $N(C_a+I)$, which implies that $q=0$, and then $f=0$. 

\underline{Assertion 3.:} For the left hand equality:  $f\in N(C_0-I)^\perp\cap N(C_a-I)^\perp$ is odd, $f(z)=\sum_{n\ge 0} \beta_n z^{2n+1}$, and similarly as above,
$$
C^*_{\omega_a}f(z)=\frac{z(\bar{\omega}_a-1)}{(1-\bar{\omega}_az)^2} \sum_{n\ge 0} \beta_n \left(\frac{\omega_a-z}{1-\bar{\omega}_az}\right)^{2n} \ \hbox{ is odd},
$$ 
so that $D(z)=\frac{1}{(1-\bar{\omega}_az)^2} \sum_{n\ge 0} \beta_n\left(\frac{\omega_a-z}{1-\bar{\omega}_az}\right)^{2n}$ is even, and
$$
C_{\omega_a}D=\frac{(1-\bar{\omega}_az)^2}{(1-|\omega_|^2)^2} \sum_{n\ge 0} \beta_n z^{2n} \in N(C_a-I).
$$
Denote $\delta(z)=\sum_{n\ge 0} \beta_nz^{2n}$, so that $(1-\bar{\omega}_a z)^2 \delta\in N(C_a-I)$.

Note that $f(z)=z\delta(z)$
Then we have 
$$
(1-\bar{\omega}_az)^2\delta=(1+(\bar{\omega}_az)^2)\delta-2\bar{\omega}_a f
$$
is an orthogonal sum: the left hand term is even and the right hand term is odd. One the other hand, rewriting this equality, we have
$$
(1+(\bar{\omega}_az)^2)\delta=(1-\bar{\omega}_az)^2\delta+2\bar{\omega}_a f
$$
is also and orthogonal sum: the left hand term belongs to $N(C_a-I)$ and the right hand term is orthogonal to $N(C_a-I)$. Then we have
$$
 \|(1-\bar{\omega}_az)^2\delta\|^2=\|(1+\bar{\omega}_az)^2)\delta\|^2+\|2\bar{\omega}_a f\|^2\  \hbox{ and }\  \|(1+(\bar{\omega}_az)^2)\delta\|^2=\|(1-\bar{\omega}_az)^2\delta\|^2+\|2\bar{\omega}_a f\|^2.
$$
These imply that $f=0$.

The right hand equality: let $f \in N(C_a-I)^\perp$ be even, and suppose first that $f(0)=0$. Then  $f(z)=\sum_{n\ge 1} \alpha_n z^{2n}$. We proceed similarly as in the third assertion, we sketch the proof. We know that 
$$
C_{\omega_a}^*(f)(z)=\frac{z(\bar{\omega}_a^2-1)}{(1-\bar{\omega}_az)^2}\sum_{n\ge 1}\alpha_n(\frac{\omega_a-z}{1-\bar{\omega}_az})^{2n-1}
$$
is even, so that $E(z)=\frac{1}{(1-\bar{\omega}_az)^2}\sum_{n\ge 1}\alpha_n(\frac{\omega_a-z}{1-\bar{\omega}_az})^{2n-1}$ is odd. Then
$$
h(z):=C_{\omega_a}(E)(z)=\frac{(1-\bar{\omega}_az)^2}{1-|\omega_a|^2}\sum_{n\ge 1} \alpha_nz^{2n-1} \in N(C_a+I).
$$
Note that $\sum_{n\ge 1} \alpha_nz^{2n-1}=\frac{f(z)}{z}$. Then we have on one hand that
$$
(1-|\omega_a|^2)h(z)=(1+(\bar{\omega}_az)^2)\sum_{n\ge 1}\alpha_nz^{2n-1}+2\bar{\omega}_a f(z)
$$
is an orthogonal sum, the left hand summand is odd and the right hand summand is even. Thus
$$
\|(1-|\omega_a|^2)h\|^2=\|(1+(\bar{\omega}_az)^2)\sum_{n\ge 1}\alpha_nz^{2n-1}\|^2+\|2\bar{\omega}_a f\|^2.
$$
On the other hand the above also means that
$$
(1+(\bar{\omega}_az)^2)\sum_{n\ge 1}\alpha_nz^{2n-1}=(1-|\omega_a|^2)h(z)+2\bar{\omega}_a f(z)
$$
is also an orthogonal sum, the left hanf summand belongs to $N(C_a+I)$ and the right hand summand belongs to $N(C_a+I)^\perp$. Then
$$
\|(1+(\bar{\omega}_az)^2)\sum_{n\ge 1}\alpha_nz^{2n-1}\|^2=\|(1-|\omega_a|^2)h\|^2+\|2\bar{\omega}_a f\|^2.
$$
These two norm identities imply that $f=0$.
Suppose now that $f(0)\ne 0$, by considering a multiple of $f$, we may assume $f(0)=1$, i.e.,
$f(z)=1+\sum_{n\ge 1} \alpha_nz^{2b}$. Then 
$$
g(z):=C^*_{\omega_a}f(z)=\frac{1}{1-\bar{\omega}_az}+(\bar{\omega}_a-1)\frac{z}{1-\bar{\omega}_az}\sum_{n\ge 1} \alpha_n\left(\frac{\omega_a-z}{1-\bar{\omega}_az}\right)^{2n-1},
$$
which is also even. Then $g'(z)$ is odd and $g'(0)=0$. Note that
$$
g'(z)=\frac{\bar{\omega}_a}{(1-\bar{\omega}_az)^2}
+\frac{\bar{\omega}_a-1}{(1-\bar{\omega}_az)^2}\sum_{n\ge 1} \alpha_n \left(\frac{\omega_a-z}{1-\bar{\omega}_az}\right)^{2n-1}+
$$
$$
+(\bar{\omega}_a-1)(|\omega_a|^2-1)\frac{z}{(1-\bar{\omega}_az)^3}\sum_{n\ge 1} \alpha_n \left(\frac{\omega_a-z}{1-\bar{\omega}_az}\right)^{2n-2},
$$
so that
$$
0=g'(0)=\bar{\omega}_a+(\bar{\omega}_a-1)\sum_{n\ge 1}\alpha_n \omega_a^{2n-1}.
$$
Note that $f(\omega_a)=1+\sum_{n\ge 1}\alpha_n\omega_a^{2n}=1+\omega_a\sum_{n\ge 1} \alpha_n\omega_a^{2n-1}$,  i.e.,
$$
0=\bar{\omega}_a+(\bar{\omega}_a-1)\left(\frac{f(\omega_a)-1}{\omega_a}\right),
$$
or $f(\omega_a)=\displaystyle{\frac{|\omega_a|^2}{1-\bar{\omega}_a}+1}$. Since $f$ is even, $f(\omega_a)=f(-\omega_a)$, i.e., $\frac{1}{1-\bar{\omega}_a}=\frac{1}{1+\bar{\omega}_a}$, or $\omega_a=0$ (which cannot happen because $a\ne 0$). It follows that $f\equiv 0$.
\end{proof}

\begin{quest}
A natural question is wether these properties above hold for arbitrary $a\ne b\in\DD$.
\end{quest}

\begin{rem}
A straightforward computation shows that if $a\in\DD$, the unique $b\in\DD$ such that the fixed point $\omega_b$ of $\varphi_b$ (in $\DD$) is $a$ is given by $b=\frac{2a}{1+|a|^2}$. Let us denote this element by $\Omega_a$. One may iterate this computation: denote by $\Omega^2_a:=\Omega_{\Omega_a}$, and in general $\Omega_a^{n+1}:=\Omega_{\Omega_a^n}$. Then it is easy to see that
$$
 \Omega_a^{n}=a_{2^{n}-1},
$$
where $a_k\in\DD$ are the numbers obtained in Lemma \ref{lema 55}. Note that all these iterations $\Omega_a^n$ are multiples of $a$, with increasing moduli, and $\Omega_a^n\to\frac{a}{|a|}$ as $n\to\infty$. 

Moreover, it is easy to see that the sequence $a_n$ is an interpolating sequence: it consists of multiples of $\frac{1-r^{n+1}}{1+r^{n+1}}$ by the number $\frac{a}{|a|}$ of modulus one, with $r<1$. Therefore   $\Omega_a^n$ is an interpolating sequence. 
\end{rem}

\section{Geodesics between Eigenspaces of $C_a$}

Recall from the introduction the condition for the existence of a geodesic of the Grassmann manifold of $H^2$ that joins two given subspaces $\s$ and $\t$, namely, that
$$
\dim(\s\cap\t^\perp)=\dim(\t\cap\s^\perp).
$$
This condition clearly holds for $\e=N(C_0-I)$ and $\o=N(C_0+I)=\e^\perp$: both intersections are, respectively, $\e\cap\o^\perp=\e$ and $\o\cap\e^\perp=\o$, and have the same (infinite) dimension. Our first observation is that this no longer holds for $N(C_a-I)$ and $N(C_a+I)$ when $a\ne 0$:
\begin{prop}
If $0\ne a\in\DD$, then there does not exist a geodesic of the Grassmann manifold of $H^2$ joining $N(C_a-I)$ and $N(C_a+I)$.
\end{prop}
\begin{proof}
The proof follows by direct computation.
First, we claim that 
\begin{equation}\label{a par a impar 1}
N(C_a+I)\cap N(C_a-I)^\perp=\{0\}.
\end{equation}
Note that $f\in N(C_a-I)^\perp$ if and only if  $\langle f, g\rangle=0$ for all $g\in N(C_a-I)=C_{\omega_a}(\e)$, i.e., 
$$
0=\langle C_{\omega_a}^*f,g\rangle,
$$
for all $g\in\e$. This is equivalent to $C_{\omega_a}^*f\in\o$, or also  that $f\in C_{\omega_a}^*(\o)$.

Using the operator $C_{\omega_a}$, our claim (\ref{a par a impar 1}) is equivalent to
$$
\{0\}=C_{\omega_a}(N(C_a+I))\cap C_{\omega_a}C_{\omega_a}^*(\o)=\o\cap C_{\omega_a}C_{\omega_a}^*(\o),
$$ 
where the last equality follows from the fact $C_{\omega_a}(N(C_a+I))=\o$ observed before. Let $f\in\o$. Then (since $f(0)=0$)
$$
g(z)=C_{\omega_a}C_{\omega_a}^*f(z)=\frac{1-\bar{\omega}_az}{1-|\omega_a|^2}\left(f(z)-\omega_a\frac{f(z)}{z}\right)
$$
$$
=\frac{1}{1-|\omega_a|^2}\left(f(z)(1+|\omega_a|^2)-\left(\omega_a\frac{f(z)}{z}+\bar{\omega}_azf(z)\right)\right).
$$
Then, since $g$ and the first summand are odd, and the second summand is even, the second summand is  zero, which implies that $f\equiv 0$.

On the other hand, a similar computation shows that
$$
\dim\left(N(C_a-I)\cap N(C_a+I)^\perp\right)=1,
$$
which would conclude the proof. Indeed, by a similar argument as above, it suffices to show that
$$
\dim(\e\cap C_{\omega_a}C_{\omega_a}^*(\e))=1.
$$
Let $g,f$ be even functions such that
$$
g(z)=C_{\omega_a}C_{\omega_a}^*f(z)=\frac{1-\bar{\omega}_az}{1-|\omega_a|^2}\left(f(z)-\omega_a\frac{f(z)-f(0)}{z}\right)
$$
$$
=\frac{1}{1-|\omega_a|^2}\left(\left(f(z)+|\omega_a|^2(f(z)-f(0))\right)-\left(\bar{\omega}_af(z)z+\omega_a\frac{f(z)-f(0)}{z}\right)\right).
$$
It follows that 
$$
\bar{\omega}_af(z)z+\omega_a\frac{f(z)-f(0)}{z}\equiv 0,
$$
i.e., $f(z)=\displaystyle{\frac{c}{\omega_a+\bar{\omega}_az^2}}$. This implies that 
$$
\e\cap C_{\omega_a}C_{\omega_a}^*(\e)=\langle \frac{1}{\omega_a+\bar{\omega}_az^2}\rangle.
$$
\end{proof}
Note though that the orthogonal projections onto $N(C_a-I)$ and $N(C_a+I)$ are unitarilly equivalent: both subspaces are infinite dimensional and infinite co-dimensional.

Also on the negative side, the subspaces $\o$ and $N(C_a-I)$, for $a\ne 0$,  cannot be joined by a geodesic:
\begin{coro}
There exist no geodesics of the Grassmann manifold of $H^2$ joining $N(C_0+I)$ and $N(C_a+I)$, for $a\ne 0$.
\end{coro}
\begin{proof}
Note that, by Theorem  \ref{teo 54}, part 1, for $b=0$:
$$
N(C_0+I)^\perp\cap N(C_a-I)=N(C_0-I)\cap N(C_a-I)=\mathbb{C} 1;
$$
whereas by Theorem \ref{0 y a}, Assertion 3, left hand identity, we have that
$$
N(C_0+I)\cap N(C_a-I)^\perp=N(C_0-I)^\perp\cap N(C_a-I)^\perp=\{0\}.
$$
\end{proof}
On the affirmative side, a direct consequence of the results in the previous section is the existence of unique normalized geodesics of the Grassmann manifold joining $\e=N(C_0-I)$ with $N(C_a-I)$,  $\o=N(C_0+I)$ with $N(C_a+I)$, and $\e$ with $N(C_a+I)$:
\begin{coro}
Let $a\in\mathbb{D}$, $a\ne 0$. 
\begin{enumerate}
\item
There exists a unique (geodesic) curve $\delta^-_{0,a}(t)=e^{tZ^-_{0,a}}\e$ of the Grassmann manifold of $H^2$, with $(Z^-_{0,a})^*=-Z^-_{0,a}$, $Z^-_{0,a}\e\subset\o$ and $\|Z^-_{0,a}\|\le\pi/2$, such that
$$
e^{Z^-_{0,a}}\e=N(C_a-I).
$$
\item
There exists a unique (geodesic) curve $\delta^+_{0,a}(t)=e^{tZ^+_{0,a}}\e$ of the Grassmann manifold of $H^2$, with $(Z^+_{0,a})^*=-Z^+_{0,a}$, $Z^+_{0,a}\e\subset\o$ and $\|Z^+_{0,a}\|\le\pi/2$, such that
$$
e^{Z^+_{0,a}}\o=N(C_a+I).
$$
\item
There exists a unique (geodesic) curve $\delta^{+,-}_{0,a}(t)=e^{tZ^{+,-}_{0,a}}\e$ of the Grassmann manifold of $H^2$, with $(Z^{+,-}_{0,a})^*=-Z^{+,-}_{0,a}$, $Z^{+,-}_{0,a}\o\subset\e$ and $\|Z^{+,-}_{0,a}\|\le\pi/2$, such that
$$
e^{Z^{+,-}_{0,a}}\o=N(C_a-I).
$$
\end{enumerate}
\end{coro}
\begin{proof}
\begin{enumerate}

\item
Follows from assertion 1 in Theorem \ref{0 y a}.
\item
Follows from assertion 2 in Theorem \ref{0 y a}.
\item
$$
N(C_0-I)\cap N(C_a+I)^\perp=\{0\},
$$
is the right hand side of assertion 2 in Theorem \ref{0 y a}. 
$$
N(C_0-I)^\perp\cap N(C_a+I)=N(C_0+I)\cap N(C_a+I)=\{0\},
$$ 
is part 2. of Theorem \ref{teo 54} for $b=0$.
\end{enumerate}
\end{proof}

\bigskip

{\bf Data availability}

\bigskip

We do not analyse or generate any datasets, because our work proceeds within a theoretical and mathematical approach.

Esteban Andruchow \\
Instituto de Ciencias,  Universidad Nacional de Gral. Sarmiento,
\\
J.M. Gutierrez 1150,  (1613) Los Polvorines, Argentina
\\ 
and Instituto Argentino de Matem\'atica, `Alberto P. Calder\'on', CONICET, 
\\
Saavedra 15, 3er. piso,
(1083) Buenos Aires, Argentina.
\\
e-mail: eandruch@ungs.edu.ar

\bigskip

Gustavo Corach\\
Instituto Argentino de Matem\'atica, `Alberto P. Calder\'on', CONICET,
\\
Saavedra 15, 3er. piso, (1083) Buenos Aires, Argentina,
\\
e-mail: gcorach@gmail.com
\bigskip

L\'azaro Recht \\
Instituto Argentino de Matem\'atica, `Alberto P. Calder\'on', CONICET,
\\
Saavedra 15, 3er. piso, (1083) Buenos Aires, Argentina,
\\
e-mail: lrecht@gmail.com


\begin{thebibliography}{XX}

\bibitem{ando} Ando, T., Unbounded or bounded idempotent operators in Hilbert space. Linear Algebra Appl. 438 (2013), no. 10, 3769--3775. 

\bibitem{p-q} Andruchow, E., Operators which are the difference of two projections,  J. Math. Anal. Appl. 420 (2014), no. 2, 1634-1653.

\bibitem{berkson} Berkson, E., Composition operators isolated in the uniform operator topology. Proc. Amer. Math. Soc. 81 (1981), no. 2, 230--232. 

 

\bibitem{brown} Brown, L. G., The rectifiable metric on the set of closed subspaces of Hilbert space. Trans. Amer. Math. Soc. 337 (1993), no. 1, 279--289. 

\bibitem{chung} Chung, K. Y., Subspaces and graphs. Proc. Amer. Math. Soc. 119 (1993), no. 1, 141--146. 

\bibitem{cpr} Corach, G.; Porta, H.; Recht, L., The geometry of spaces of projections in $C^*$-algebras, Adv. Math. 101 (1993), no. 1, 59--77.

\bibitem{cowen} Cowen, C. C., Linear fractional composition operators on $H^2$. Integral Equations Operator Theory 11 (1988), no. 2, 151--160. 

\bibitem{cowenmccluer} Cowen, C. C.; MacCluer, B. D.,  Composition operators on spaces of analytic functions. Studies in Advanced Mathematics. CRC Press, Boca Raton, FL, 1995. 


\bibitem{dixmier} Dixmier, J., Position relative de deux vari\'et\'es lin\'eaires ferm\'ees dans un espace de Hilbert,  Revue Sci. 86 (1948), 387-399.

\bibitem{douglas} Douglas, R. G.,  Banach algebra techniques in operator theory. Second edition. Graduate Texts in Mathematics, 179. Springer-Verlag, New York, 1998.


\bibitem{halmos} Halmos, P. R., Two subspaces, Trans. Amer. Math. Soc. 144 (1969), 381--389.


\bibitem{kovarik} Kovarik, Z. V. Manifolds of linear involutions. Linear Algebra Appl. 24 (1979), 271--287. 

\bibitem{phillips} Phillips, N. C., The rectifiable metric on the space of projections in a $C^*$-algebra. Internat. J. Math. 3 (1992), no. 5, 679--698.

\bibitem{pr} Porta, H.; Recht, L., Minimality of geodesics in Grassmann manifolds, Proc. Amer. Math. Soc. 100 (1987), 464--466.



\bibitem{rosenblum} M. Rosenblum, Self-adjoint Toeplitz operators and associated orthonormal functions. Proc. Amer. Math. Soc. 13 (1962), 590--595.

\bibitem{wilkins} Wilkins, D. R., The Grassmann manifold of a $C^*$-algebra. Proc. Roy. Irish Acad. Sect. A 90 (1990), no. 1, 99--116.

\bibitem{zemanek} Zem\'anek, J., Idempotents in Banach algebras. Bull. London Math. Soc. 11 (1979), no. 2, 177--183.

\end{thebibliography}
\end{document}